\newif\ifarxiv
\arxivtrue  

\documentclass[twoside]{article}

\newcommand{\figwidth}{0.5\linewidth}
\newcommand{\sidebysidefigwidth}{0.49\linewidth}

\usepackage{amsbsy}
\usepackage{amsfonts}
\usepackage{amsmath}
\usepackage{amssymb}
\usepackage{booktabs}
\usepackage{makecell,multirow}

\usepackage{hyperref}
\usepackage{url,intmacros,graphicx}
\usepackage[numbers,sort]{natbib}
\usepackage[small]{caption}
\usepackage[toc]{appendix}

\def\BibTeX{{\rm B\kern-.05em{\sc i\kern-.025em b}\kern-.08em
    T\kern-.1667em\lower.7ex\hbox{E}\kern-.125emX}}

\newtheorem{corollary}{Corollary}
\newtheorem{definition}{Definition}
\newtheorem{proposition}{Proposition}

\def\qed{\unskip\kern10pt{\unitlength1pt\linethickness{.4pt}\framebox(6,6){}}}
\newenvironment{proof}{{\it Proof:\/}}{\hfill\qed}

\ifarxiv
\newtheorem{lemma}{Lemma}
\newtheorem{theorem}{Theorem}

\newcommand{\gi}{I}
\newcommand{\gilep}{\lep{I}}
\newcommand{\girep}{\rep{I}}
\newcommand{\fenclosure}{F}

\else
\newtheorem{lemma}{Lemma}[section]
\newtheorem{theorem}{Theorem}[section]

\newcommand{\gi}{\z}  
\newcommand{\gilep}{\lep{z}}
\newcommand{\girep}{\rep{z}}
\newcommand{\fenclosure}{\tilde{\f}}

\newcommand{\email}[1]{\\ \small{\url{#1}} \\}
\newcommand{\institution}[1]{\\ \parbox{3.0in}{\small{#1}}}
\newcommand{\keywords}[1]{\small\textbf{Keywords: }#1}
\newcommand{\AMSsubj}[1]{\noindent\textbf{AMS subject classifications: }#1}
\newcommand\whenaccepted{Submitted: (insert date);
                         Revised: (insert date);
                         Accepted:(insert date).}

\fi

\renewcommand{\dd}{\mathrm{d}}
\newcommand{\eqdef}{\triangleq}
\newcommand{\ignore}[1]{}

\newcommand{\integers}{\mathbb{Z}}

\newcommand{\paren}[1]{\left(#1\right)}

\newcommand{\reals}{\mathbb{R}}

\renewcommand{\set}[1]{\left\{#1\right\}}

\newcommand{\lep}[1]{\underline{#1}}
\newcommand{\Ias}{\tilde \gi}
\newcommand{\Idb}{\gi^{\mathrm{Lagrange}}}
\newcommand{\intervals}{\mathbb{IR}}
\newcommand{\Istar}{\gi^*}
\newcommand{\Istarlep}{\gilep^*}
\newcommand{\Istarrep}{\girep^*}
\newcommand{\rep}[1]{\overline{#1}}

\newcommand{\weight}{\mu}

\newtheorem{innercustomlemma}{Lemma}
\newenvironment{customlemma}[1]
  {\renewcommand\theinnercustomlemma{#1}\innercustomlemma}
  {\endinnercustomlemma}

\newtheorem{innercustomproposition}{Proposition}
\newenvironment{customproposition}[1]
  {\renewcommand\theinnercustomproposition{#1}\innercustomproposition}
  {\endinnercustomproposition}

\newtheorem{innercustomthm}{Theorem}
\newenvironment{customthm}[1]
  {\renewcommand\theinnercustomthm{#1}\innercustomthm}
  {\endinnercustomthm}



\ifarxiv
\title{
Sharp Taylor Polynomial Enclosures\\ in One Dimension\footnote{This work first appeared as chapter 2 of \small{\url{https://arxiv.org/pdf/2212.11429v1.pdf}}.}
}
\author{Matthew Streeter and Joshua V. Dillon\\ 
\vspace{-.2cm} \\
\small{{\tt \href{mailto:mstreeter@google.com}{mstreeter@google.com}, \href{mailto:jvdillon@google.com}{jvdillon@google.com}}}}
\date{}
\else
\title{Sharp Taylor Polynomial Enclosures\\ in One Dimension\footnote{\whenaccepted}}

\author{Matthew Streeter and Joshua V. Dillon
\institution{Google Research, Mountain View, CA 94043, USA}
\\ \small{\url{mstreeter@google.com}, \url{jvdillon@google.com}} \\ }

\date{}
\fi

\begin{document}

\maketitle

\begin{abstract}
It is often useful to have polynomial upper or lower bounds on a one-dimensional function that are valid over a finite interval, called a \emph{trust region}.  A classical way to produce polynomial bounds of degree $k$ involves bounding the range of the $k$th derivative over the trust region, but this produces suboptimal bounds.  We improve on this by deriving \emph{sharp} polynomial upper and lower bounds for a wide variety of one-dimensional functions.  We further show that sharp bounds of degree $k$ are at least $k+1$ times tighter than those produced by the classical method, asymptotically as the width of the trust region approaches zero.
We discuss how these sharp bounds can be used in majorization-minimization optimization, among other applications.
\end{abstract}
\ifarxiv
\else

\AMSsubj{41A10}
\fi

\section{Introduction} \label {sec:intro}

Taylor polynomials are among the most useful tools in science and engineering.  However, their usefulness is limited by the fact that they provide only a local approximation to a function, with no guarantees about the approximation's accuracy.  In many applications, it is preferable to derive polynomial upper and lower bounds that are valid over a small \emph{trust region} \cite{berz1999new,makino1996remainder,makino1998rigorous}.  
We will see that such bounds can be obtained by modifying the maximum-degree coefficient of the Taylor polynomial.  However, prior to this work, it was not known how to compute optimal bounds even for rather simple functions.

In this paper we will develop theory that lets us compute optimal \emph{Taylor polynomial enclosures}.  Like Taylor polynomials, a Taylor polynomial enclosure is centered at a point $x_0$, and is tight at this point.  However, rather than approximating a function $f$, a Taylor polynomial enclosure provides upper and lower bounds that are valid over some trust region $[a, b]$.  A \emph{sharp} Taylor polynomial enclosure provides the tightest possible bounds.  Figure~\ref{fig:example_enclosure} illustrates the sharp quadratic Taylor polynomial enclosure for the function $f(x) = \frac 3 2 \exp(3 x) - 25 x^2$, at $x_0 = \frac 1 2$, for the trust region $[0, 1]$.

\begin{figure}[h]
\begin{center}
\includegraphics[width=\figwidth]{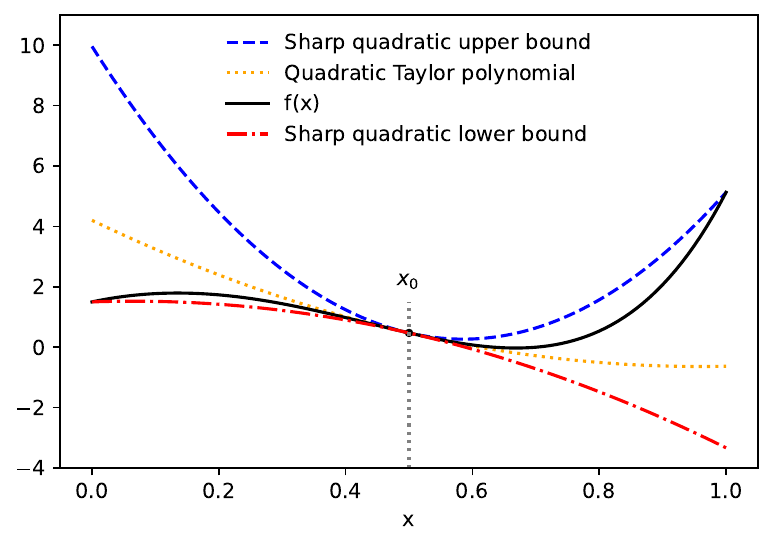}
\caption{Sharp quadratic upper and lower bounds for the function $f(x) = \frac 3 2 \exp(3 x) - 25 x^2$, centered at $x_0 = \frac 1 2$, and valid over the interval $[0, 1]$.}
\label{fig:example_enclosure}
\end{center}
\end{figure}

Observe that the upper bound in Figure~\ref{fig:example_enclosure} is tight at two points: $x = x_0$, and $x = 1$, the maximum value in the trust region.  Likewise the lower bound is tight at $x = x_0$ and at $x = 0$, the minimum value in the trust region.  Having the lower and upper bounds be tight at two (or more) points is a general feature of sharp quadratic bounds \cite{van2005algorithms}. 

To further illustrate the idea, suppose we are interested in the behavior of the $\exp$ function around the point $x_0 = \frac 1 2$.  Using the second degree Taylor polynomial, we may obtain the approximation
\begin{equation}
  \exp(x) \approx \sqrt{e} + \sqrt{e} \paren { x - \frac 1 2 } + 0.82436 \paren { x - \frac 1 2 }^2 \quad \mbox { for } x \approx \frac 1 2 .
\end{equation}
The approximation is tight at $x = \frac 1 2$, but its error grows exponentially for $x \gg x_0$.

Because $\exp(x)$ grows faster than any polynomial, no polynomial upper bound can hold for all $x \in \reals$.  However, if we only require the bounds to hold for $x$ belonging to some finite interval, the theory developed in this paper will allow us to compute the tightest polynomial bounds that are valid over that interval.  For example, we will be able to show
\begin{equation} \label{eq:exp_sharp_quadratic}
  \exp(x) \in \sqrt{e} + \sqrt{e} \paren { x - \frac 1 2 } + \left [ 0.70255, 1.4522 \right ] \paren { x - \frac 1 2 }^2 \quad \mbox { for } x \in [0, 2].\footnote{Recall that the product of an interval $\gi = [\gilep, \girep]$, and a scalar $\alpha$ is defined as $\gi \alpha \eqdef \set{z \alpha: z \in \gi} = [ \min \set { \gilep \alpha, \girep \alpha} , \max \set { \gilep \alpha, \girep \alpha} ]$.}
\end{equation}

The expression on the right hand side of \eqref{eq:exp_sharp_quadratic} defines a Taylor polynomial enclosure, which is similar to a Taylor polynomial except that the maximum-degree coefficient is an interval rather than a scalar.  The enclosure is called \emph{sharp} if this interval is as narrow as possible, as is the case in \eqref{eq:exp_sharp_quadratic}.  A Taylor polynomial enclosure can be thought of as a function that returns an interval, or equivalently, as a pair of real-valued functions, one of which returns a lower bound and one of which returns an upper bound.

\subsection{Contributions}

The primary contributions of this work are as follows:
\begin{itemize}
  \item In \S\ref{sec:sharp}, we derive \emph{sharp} Taylor polynomial enclosures of arbitrary degree $k$ for functions whose $k$th derivative is monotonically increasing or decreasing (such as $\exp$ or $\log$).  We also derive sharp \emph{quadratic} Taylor polynomial enclosures for functions whose Hessian is even-symmetric (such as $\mathrm{softplus}$, $\mathrm{relu}$, and a number of other commonly-used neural network activation functions). 
   \item In \S\ref{sec:comparison} we show that sharp Taylor polynomial enclosures of degree $k$ offer at least a factor $k+1$ improvement over a classical baseline \cite{jaulin2001interval,hansen1979global,hansen2003global,moore1966interval},
asymptotically as the width of the trust region approaches zero.  The classical baseline obtains (non-sharp) Taylor polynomial enclosures using the range of the $k$th derivative, as discussed further in \S\ref{sec:spoly_baseline}.
   \item In \S\ref{sec:arbitrary} we present several methods for computing Taylor polynomial enclosures for \emph{arbitrary} one-dimensional functions.  The enclosures produced by these methods are not sharp in general, but become sharp \emph{asymptotically} as the width of the trust region approaches zero.  
\end{itemize}

\subsection{Applications} \label {sec:applications}

The motivation for this work was to derive tighter majorizers for use in majorization-minimization (MM) optimization \cite{de2016block,lange2016mm}.  MM is a class of optimization methods that iteratively reduce a loss by minimizing a locally-tight upper bound on the loss, called a \emph{majorizer}.  On iteration $t$, an MM optimizer minimizes a majorizer that is tight at some point $x_t$ in order to obtain a point $x_{t+1}$ with lower (or in trivial cases, equal) loss, as illustrated in Figure~\ref{fig:mm}.  This process can be repeated until a local minimum is reached.

\begin{figure}[h]
\begin{center}
\includegraphics[width=\figwidth]{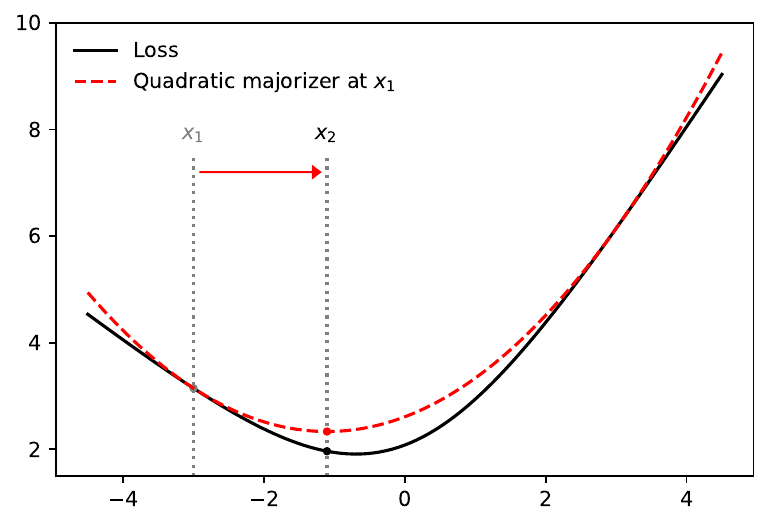}
\caption{Majorization-minimization (MM) optimization.  Starting at a point $x_1$, an MM optimizer computes an upper bound (majorizer) that is tight at $x_1$, then minimizes the upper bound to obtain a point $x_2$ with lower loss.  The process can be repeated until a local minimum is reached.}
\label{fig:mm}
\end{center}
\end{figure}

By design, the upper bound provided by a Taylor polynomial enclosure is a majorizer, and our theory therefore provides the tightest possible majorizers (of a given polynomial degree) for certain one-dimensional functions.  At first glance, this may not seem to be a practically useful result given that our theory does not apply directly to most losses that arise in practice.  But in fact, our work can serve as the basis of a recursive algorithm that computes majorizers for much more complex losses.  In concurrent work \cite{streeter2023automatically}, we define such an algorithm and use it to automatically derive MM optimizers for training deep neural networks.  This yields neural network optimizers that converge to a local minimum from any starting point, without the need for time-consuming tuning of learning rate hyperparameters.

Taylor polynomial enclosures can also be applied to problems to which \emph{Taylor models} \cite{makino1996remainder} have been previously applied, such as verified numerical integration \cite{berz1999new} and verified global optimization \cite{makino2005verified}.  We also believe our sharp Taylor polynomial enclosures could be used to derive tighter Taylor models, but defer exploration of this possibility to future work.

In addition to these practical possibilities, we see our work as providing a large number of mathematical inequalities of the type that are often derived by hand for use in proofs.  For example, for $x \in [-1, 1]$, it implies $1 + x + \frac 1 2 x^2 \le \exp(1 + x) \le 1 + x + \frac{e^2 - 3}{4} x^2$.  Additional examples are given in Table~\ref{tab:spoly_monotone_corollaries}.

\section{Related Work} \label{sec:spoly_related}

Deriving polynomial upper and lower bounds on functions has been the subject of much research.  A classical approach, discussed in detail in \S\ref{sec:spoly_baseline}, uses interval arithmetic to obtain bounds on the range of the $k$th derivative of $f$ over $[a, b]$, then uses these bounds to construct a degree $k$ Taylor polynomial enclosure \cite{jaulin2001interval,hansen1979global,hansen2003global,moore1966interval}.

A different type of polynomial upper and lower bounds are provided by \emph{Taylor models} \cite{makino1996remainder,makino1999efficient,makino2001higher,makino2003taylor}.  In Taylor models, the upper and lower bounds differ only in the constant term (whereas in a Taylor polynomial enclosure, the upper and lower bounds differ only in the maximum-degree coefficient).  This means that Taylor models give upper and lower bounds that are not tight at $x_0$, as illustrated in Figure~\ref{fig:tm_vs_tpe}, which rules out the application to MM optimization discussed in \S\ref{sec:applications}.

\begin{figure}[h]
\begin{center}
\includegraphics[width=\figwidth]{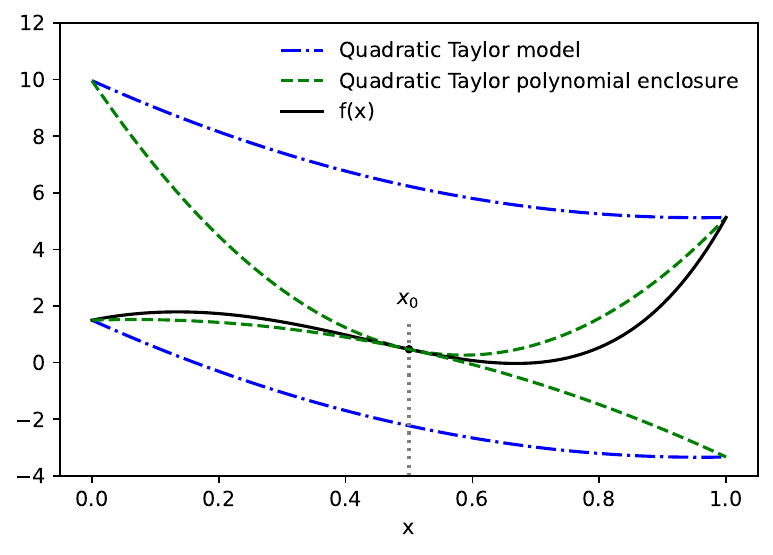}
\caption{The difference between Taylor models and Taylor polynomial enclosures.  Only the latter provide bounds that are tight at $x_0$, which is necessary for the application to majorization-minimization optimization.}
\label{fig:tm_vs_tpe}
\end{center}
\end{figure}

The work in this paper was inspired by a paper by de Leeuw and Lange \cite{de2009sharp}, who introduced the notion of sharp quadratic majorizers for one-dimensional functions, and derived closed-form majorizers for a number of common functions.  Using our terminology, deriving a sharp quadratic majorizer for a function $f: \reals \to \reals$ is equivalent to deriving the upper bound of a degree $2$ sharp Taylor polynomial enclosure, over the trust region $[a, b] = [-\infty, \infty]$.  Thus, our results generalize those of \cite{de2009sharp} in three ways: by deriving polynomial bounds of arbitrary degree, by providing both upper and lower bounds, and by deriving the tightest bounds that are valid over an arbitrary trust region.  We note that the use of a finite trust region is essential for computing Taylor polynomial enclosures for functions such as $\exp$, which asymptotically grow faster than any fixed degree polynomial.

A different approach to obtaining tighter bounds for one-dimensional functions was taken in \cite{de2012polynomial}.  Observing that higher-degree polynomials provide tighter bounds near $x_0$ but that lower-degree polynomials can provide tighter bounds far from $x_0$, \cite{de2012polynomial} derives piecewise-polynomial bounds that attempt to get the best of both worlds.  This technique is complementary to our work, and can be applied on top of it.

Extending the results of de Leeuw and Lange in a different direction, Browne and McNicholas \cite{browne2015multivariate} derived sharp quadratic majorizers for multivariate functions.  Generalizing their results to take into account a trust region would be an interesting area of future work.

\section{Definitions and Notation} \label{sec:spoly_definitions}


For a $j$-times differentiable function $f: \reals \to \reals$, and a scalar $x_0 \in \reals$, we denote the degree $j$ Taylor polynomial of $f$ at $x_0$ by
\begin{equation}
  T_j(x; f, x_0) \eqdef \sum_{i=0}^j \frac {1} {i!} f^{(i)}(x_0) (x - x_0)^i.
\end{equation}
We denote the corresponding remainder term by
\begin{equation}
  R_j(x; f, x_0) \eqdef f(x) - T_j(x; f, x_0) .
\end{equation}
Consistent with this definition, we define $T_{-1}(x; f, x_0) \eqdef 0$ and $R_{-1}(x; f, x_0) \eqdef f(x)$.

We now define notation that is standard in interval analysis \cite{moore1966interval,jaulin2001interval}.
The set of closed real intervals is denoted by $\intervals \eqdef \set { [a, b]: a, b \in \reals, a \le b }$.
\ifarxiv
We use capital letters for intervals.
\else
We use bold lower case symbols for real intervals.
\fi
The product of an interval $\gi = [\gilep, \girep] \in \intervals$ and scalar $\alpha \in \reals$ is defined as $\gi \alpha \eqdef  \set{z \alpha: z \in \gi}$, and $\alpha \gi$ is defined identically.  Observe that if $\alpha \ge 0$, $\gi \alpha = [\gilep \alpha, \girep \alpha]$, while if $\alpha \le 0$, $\gi \alpha = [\girep \alpha, \gilep \alpha]$.

With this notation in place, we can now define Taylor polynomial enclosures.

\begin{definition} [Taylor polynomial enclosure] \label{def:tif}
If, for some $(k-1)$-times differentiable function $f: \reals \to \reals$, scalar $x_0 \in \reals$, and trust region $[a, b] \in \intervals$, an interval $\gi \in \intervals$ satisfies
\[
  R_{k-1}(x; f, x_0) \in \gi (x - x_0)^k \quad \forall x \in [a, b]
\]
then the function $\fenclosure: \reals \to \intervals$ defined by
\[
  \fenclosure(x) = T_{k-1} (x; f, x_0) + \gi (x - x_0)^k
\]
is a \emph{degree $k$ Taylor polynomial enclosure} of $f$ at $x_0$ over $[a, b]$.
\end{definition}
Observe that if $\fenclosure$ is a degree $k$ Taylor polynomial enclosure of $f$ at $x_0$ over $[a, b]$, then $f(x) \in \fenclosure(x)$ for all $x \in [a, b]$.

Of all the intervals $\gi$ that define a Taylor polynomial enclosure, there is a unique narrowest interval, whose end points are given by the following proposition.

\newcommand{\propuniquesharp}{
For a $(k-1)$-times differentiable function $f: \reals \to \reals$, scalar $x_0 \in \reals$, and trust region $[a, b]$, an interval $\gi = [\gilep, \girep] \in \intervals$ defines a degree $k$ Taylor polynomial enclosure of $f$ at $x_0$ over $[a, b]$ if any only if
\[
  \gilep \le \inf_{x \in [a, b] \setminus \set{x_0}} \set { \frac{ R_{k-1}(x; f, x_0) }{ (x - x_0)^{k} } }
\]
and
\[
  \girep \ge \sup_{x \in [a, b] \setminus \set{x_0}} \set { \frac{ R_{k-1}(x; f, x_0) }{ (x - x_0)^{k} } }.
\]
}
\newcommand{\propuniquesharpproof}{
By Definition~\ref{def:tif}, $\gi$ defines a Taylor polynomial enclosure of $f$ at $x_0$ over $[a, b]$ iff.\ for all $x \in [a, b]$,
\begin{equation} \label {eq:enclosure}
  R_{k-1}(x; f, x_0) \in \gi (x - x_0)^{k}.
\end{equation}
Because $R_{k-1}(x_0; f, x_0) = 0$, \eqref{eq:enclosure} holds trivially for $x = x_0$.  For $x \neq x_0$, multiplying both sides by $\frac {1} {(x - x_0)^{k}}$ and using \eqref{eq:scalar_interval_product_associative} shows that \eqref{eq:enclosure} is equivalent to $\frac{R_{k-1}(x; f, x_0)} {(x - x_0)^{k}} \in \gi$.  Thus, $\gi$ defines a Taylor polynomial enclosure of $f$ at $x_0$ iff.\  $\frac{R_{k-1}(x; f, x_0)} {(x - x_0)^{k}} \in \gi$ for all $x \in [a, b] \setminus \set{x_0}$, which is equivalent to the two inequalities listed in the proposition.
}
\begin{proposition} \label{prop:unique_sharp}
\propuniquesharp
\end{proposition}

Proposition~\ref{prop:unique_sharp} follows from Definition~\ref{def:tif} using elementary properties of interval arithmetic; see Appendix A for a formal proof.


Proposition~\ref{prop:unique_sharp} shows that there is a unique sharp Taylor polynomial enclosure, defined formally in Definition~\ref{def:spoly}.

\begin{definition} [Sharp Taylor polynomial enclosure] \label{def:spoly}
For a $(k-1)$-times differentiable function $f: \reals \to \reals$, scalar $x_0 \in \reals$, and trust region $[a, b]$, the degree $k$ \emph{sharp Taylor polynomial enclosure} of $f$ at $x_0$ over $[a, b]$ is the Taylor polynomial enclosure defined by the interval
\[
  \Istar_k(f, x_0, [a, b]) \eqdef \left [ \inf_{x \in [a, b] \setminus \set{x_0}} \set { \frac{ R_{k-1}(x; f, x_0) }{ (x - x_0)^k } }, \sup_{x \in [a, b] \setminus \set{x_0}} \set { \frac{ R_{k-1}(x; f, x_0) }{ (x - x_0)^k } } \right ].
\]
\end{definition}
We will use $\Istarlep_k$ and $\Istarrep_k$ to denote the left and right endpoints, respectively, of $\Istar_k$.

\subsection{Discussion of Definitions when $k$ is Odd}

It is worth noting that the upper and lower bounds given by a Taylor polynomial enclosure are only polynomials when $k$ is even.  When $k$ is odd, the bounds are piecewise polynomials, with one piece for $x \le x_0$ and one piece for $x \ge x_0$.  This stems from the fact that if $\gi = [\gilep, \girep]$ defines a degree $k$ Taylor polynomial enclosure, then if $k$ is even, $\gi (x - x_0)^k = [ \gilep (x - x_0)^k, \girep(x - x_0)^k ]$ for all $x$, while if $k$ is odd this equality holds only for $x \ge x_0$, and for $x < x_0$ we instead have $\gi (x - x_0)^k = [ \girep (x - x_0)^k, \gilep(x - x_0)^k ]$.

\begin{figure}[h]
\begin{center}
\includegraphics[width=\sidebysidefigwidth]{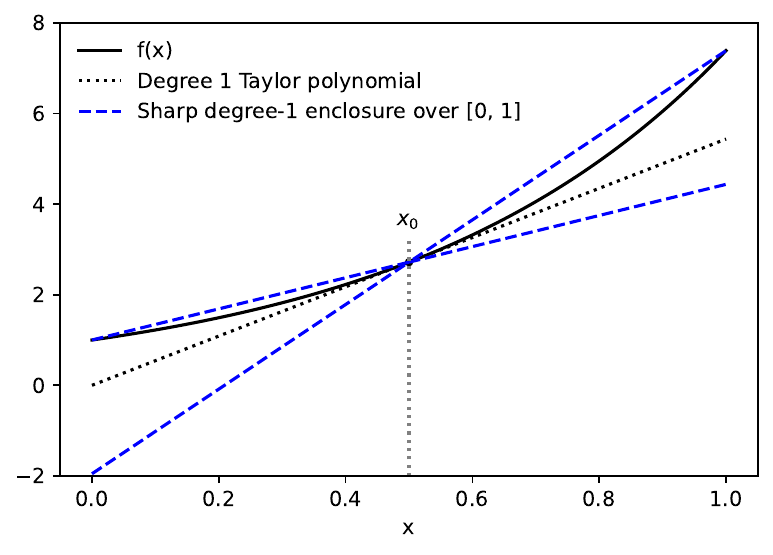}
\includegraphics[width=\sidebysidefigwidth]{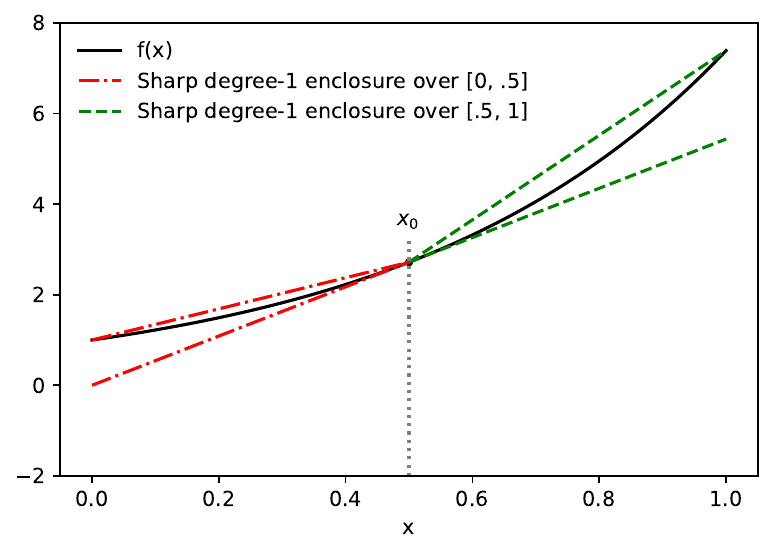}
\caption{Sharp degree-1 enclosures for the convex function $f(x) = e^{2x}$ at $x_0 = .5$. 
}
\label{fig:degree_1}
\end{center}
\end{figure}

One possibly surprising consequence of this is that if $f$ is convex, the lower bound given by the sharp degree 1 Taylor polynomial enclosure may be looser than the lower bound given by the degree 1 Taylor polynomial, as illustrated in Figure~\ref{fig:degree_1} (left).  
This happens because the slopes of the four line segments that comprise the upper and lower bounds are defined by just two parameters, which for odd $k$ are shared between the upper and lower bounds, and in Figure~\ref{fig:degree_1} (left) it is not possible to make the lower bound tighter without making the upper bound invalid.

At first, this might seem like a flaw in the definition of sharp Taylor polynomial enclosures.  However, this surprising behavior can be avoided by always making $x_0$ one of the endpoints of the trust region if $k$ is odd.  As illustrated in Figure~\ref{fig:degree_1} (right), computing separate sharp enclosures for the trust regions $[a, x_0]$ and $[x_0, b]$ and combining them yields a piecewise-linear lower bound that is at least as tight as any affine lower bound that is exact at $x_0$.

A tempting alternative definition, which coincides with Definition~\ref{def:tif} when $k$ is even, would define a Taylor polynomial enclosure by a pair $(\alpha, \beta)$ such that $\alpha (x - x_0)^k \le R_{k-1}(x; f, x_0) \le \beta (x - x_0)^k$, thereby avoiding the parameter sharing between upper and lower bounds.  The downside of this approach is that such $\alpha$ and $\beta$ do not typically exist when $k$ is odd.  In particular for $k = 1$, it can be shown that such an $(\alpha, \beta)$ pair only exists if $f$ is affine.  In contrast, the use of interval arithmetic in Definition~\ref{def:tif} ensures that degree $k$ Taylor enclosures exist for arbitrary analytic functions.

\section{Baseline: Non-Sharp Enclosures via Bounded Derivatives} \label{sec:spoly_baseline}

Before attempting to derive \emph{sharp} Taylor polynomial enclosures, we review a well-known way to derive (possibly non-sharp) Taylor polynomial enclosures using bounds on derivatives.

Suppose we desire a degree $k$ Taylor polynomial enclosure of a function $f: \reals \to \reals$, valid over $[a, b]$, and tight at $x_0 \in [a, b]$, and further suppose that $f$ is $k$ times differentiable (rather than merely $k-1$ times differentiable, as required by Definition~\ref{def:tif}).  Then, using the Lagrange form of the Taylor remainder series and the mean value theorem \cite{apostol1991calculus}, we can show that for any $x \in \reals$ there exists a $y$ in the closed interval between $x_0$ and $x$, such that
\begin{equation}
  R_{k-1}(x; f, x_0) = \frac{1}{k!} f^{(k)}(y) (x - x_0)^{k}.
\end{equation}
Because $x_0, x \in [a, b]$ we have $y \in [a, b]$ also, and thus
\begin{equation} \label{eq:lagrange_remainder}
  R_{k-1}(x; f, x_0) \in \frac{1}{k!} \left [\inf_{y \in [a, b]} \set { f^{(k)}(y) }, \sup_{y \in [a, b]} \set { f^{(k)}(y) } \right ] (x - x_0)^{k}.
\end{equation}
Equation \eqref{eq:lagrange_remainder} thus defines a Taylor polynomial enclosure of $f$ at $x_0$ over $[a, b]$.

In special cases, the end points of the interval in \eqref{eq:lagrange_remainder} can be computed in closed form.  For example, if $f^{(k)}$ is monotonically increasing, the interval in \eqref{eq:lagrange_remainder} simplifies to $[f^{(k)}(a), f^{(k)}(b)]$.
Even in these special cases, however, the resulting Taylor polynomial enclosure will not be sharp in general.  For example, if used to compute a Taylor polynomial enclosure of $\exp$ at $x_0 = \frac 1 2$ over the trust region $[0, 2]$, this method yields the Taylor polynomial enclosure
\begin{equation} \label{eq:exp_non_sharp}
  \exp(x) \in \sqrt{e} + \sqrt{e} \paren { x - \frac 1 2 } + \left [ 0.5, 3.6945 \right ] \paren { x - \frac 1 2 }^2 \quad \mbox { for } x \in [0, 2].
\end{equation}
In contrast, as shown in \eqref{eq:exp_sharp_quadratic}, the sharp Taylor polynomial enclosure is defined by the interval $[0.70255, 1.4522]$, which gives significantly tighter upper and lower bounds.
A more extreme example of non-sharpness arises when computing a quadratic upper bound for the function $\mathrm{relu}(x) \eqdef \max \set{x, 0}$.  As long as $0 \in [a, b]$, the interval in \eqref{eq:lagrange_remainder} is $[0, \infty]$, leading to a vacuous upper bound.  In contrast, the sharp quadratic Taylor polynomial enclosure will give a finite upper bound as long as $x_0 \neq 0$.

If $f^{(k)}$ is not known to have any special properties (such as being monotonically increasing or decreasing), an interval that encloses the interval on the right hand side of \eqref{eq:lagrange_remainder} can be obtained by first deriving an expression for $f^{(k)}(y)$ (e.g., using automatic differentiation), and then evaluating this expression using interval arithmetic \cite{jaulin2001interval,hansen1979global,hansen2003global,moore1966interval}.  Given an interval that contains $y$ (here $[a, b]$), this procedure provides an interval that is guaranteed to contain $f^{(k)}(y)$ for all $y \in [a, b]$, and hence contains the interval on the right hand side of \eqref{eq:lagrange_remainder}.  Note, however, that the resulting interval depends on the expression for $f^{(k)}(y)$, and may be much wider than the interval in \eqref{eq:lagrange_remainder}, introducing an additional source of non-sharpness.


\section{Sharp Taylor Polynomial Enclosures} \label{sec:sharp}

We now develop theory that lets us compute sharp Taylor polynomial enclosures of arbitrary degree for one-dimensional functions with certain properties:
\begin{itemize}
  \item In \S\ref{sec:monotone}, we show how to compute sharp Taylor polynomial enclosures for functions whose $k$th derivative is monotonically increasing or decreasing (such as $\exp$ or $\log$).
  \item In \S\ref{sec:even_symmetric}, we show how to compute sharp \emph{quadratic} Taylor polynomial enclosures for functions whose Hessian is even-symmetric (such as $\mathrm{softplus}$, $\mathrm{relu}$, and a number of other commonly-used neural network activation functions). 
\end{itemize}

\subsection{Functions with Monotone Derivative} \label{sec:monotone}

\ignore{
Applying a monotonically increasing function to an interval is straightforward: if $f: \reals \to \reals$ is monotonically increasing, then $f([a, b]) = [f(a), f(b)]$.  Similarly, if $f: \reals \to \reals$ is monotonically decreasing, $f([a, b]) = [f(b), f(a)]$.
}

We will now show that the interval defining the sharp Taylor polynomial enclosure of $f$ (at some $x_0 \in \reals$ over some trust region $[a, b] \in \intervals$) can be computed in closed form in the special case where $f^{(k)}$ is monotonically increasing or monotonically decreasing.  Among other things, this will let us compute sharp Taylor polynomial enclosures for $\exp$ over any interval, for $\log$ over any interval $[a, b]$ with $a > 0$, and for the reciprocal function over any $[a, b]$ with $0 \notin [a, b]$.

To build intuition, we will first show that if $f'(x)$ is monotonically increasing, then the degree 1 sharp Taylor polynomial enclosure of $f$ is defined by the interval
\begin{equation} \label{eq:sharp_degree_1_monotone}
  \Istar_1(f, x_0, [a, b]) = \left [ \frac { R_0(a; f, x_0) } {a - x_0}, \frac { R_0(b; f, x_0) } {b - x_0} \right ].
\end{equation}
To see this, observe that if $f'$ is monotonically increasing, then for $k = 1$ and $x \neq x_0$ we have
\begin{align}
  \frac{\dd}{\dd x} \frac{R_{k-1}(x; f, x_0)}{(x - x_0)^{k}}
  & = \frac{\dd}{\dd x} \frac { f(x) - f(x_0) } {x - x_0} \nonumber \\
  & = \frac {f'(x)}{x - x_0} - \frac{f(x) - f(x_0)} {(x - x_0)^2} \nonumber \\
  & = \frac {1} {(x - x_0)^2} \paren {  (x - x_0) f'(x) - (f(x) - f(x_0)) } \nonumber \\
  & = \frac {1} {(x - x_0)^2} \int_{y=x_0}^x f'(x) - f'(y) \dd y \nonumber \\
  & \ge 0.
\end{align}
Thus, for $k = 1$, we see that if $f'$ is monotonically increasing, then $\frac{R_{k-1}(x; f, x_0)}{(x - x_0)^{k}}$ is also monotonically increasing.  Therefore, by Definition~\ref{def:spoly}, the end points of $\Istar_{k}(f, x_0, [a, b])$ are $\inf_{x \in [a, b]} \set { \frac{R_{k-1}(x; f, x_0)}{(x - x_0)^{k}} } = \frac{R_{k-1}(a; f, x_0)}{(a - x_0)^{k}}$ and $\sup_{x \in [a, b]} \set { \frac{R_{k-1}(x; f, x_0)}{(x - x_0)^{k}} } = \frac{R_{k-1}(b; f, x_0)}{(b - x_0)^{k}}$, respectively.  Plugging in $k = 1$ then gives \eqref{eq:sharp_degree_1_monotone}.

We will generalize this result to obtain the following theorem, which allows us to compute a sharp degree $k$ Taylor polynomial enclosure of any function whose $k$th derivative is monotonically increasing or decreasing.

\begin{theorem} \label {thm:spoly_monotone}
Let $f: \reals \to \reals$ be a function that is $k$-times differentiable over an interval $[a, b]$, for $k \ge 0$.
If $f^{(k)}$ is monotonically increasing over $[a, b]$, then for $x_0 \in (a, b)$, the degree $k$ sharp Taylor polynomial enclosure of $f$ at $x_0$ over $[a, b]$ is given by $\Istar_{k}(f, x_0, [a, b]) = \left [ \frac {R_{k-1}(a; f, x_0)} {(a - x_0)^{k}},  \frac {R_{k-1}(b; f, x_0)} {(b - x_0)^{k}} \right ]$.  If $f^{(k)}$ is monotonically decreasing over $[a, b]$, we instead have $\Istar_{k}(f, x_0, [a, b]) = \left [ \frac {R_{k-1}(b; f, x_0)} {(b - x_0)^{k}},  \frac {R_{k-1}(a; f, x_0)} {(a - x_0)^{k}} \right ]$.
\end{theorem}

In the special case $k = 0$, Theorem~\ref{thm:spoly_monotone} says that if $f$ is monotonically increasing, then $\Istar_0(f, x_0, [a, b]) = [f(a), f(b)]$, recovering the trivial fact that if $f$ is monotonically increasing, then $[f(a), f(b)]$ is the narrowest interval that contains $\set{f(x): x \in [a, b]}$.  In the special case $k = 1$, Theorem~\ref{thm:spoly_monotone} recovers the result we just derived.

We now outline the proof of Theorem~\ref{thm:spoly_monotone} for arbitrary $k$.  First note that, to prove the theorem, it suffices to show that if $f^{(k)}$ is monotonically increasing (resp.\ decreasing), then $\frac{\dd}{\dd x} \frac { R_{k-1}(x; f, x_0)  } { (x - x_0)^{k} }$ is non-negative (resp.\ non-positive) for $x \in [a, b]$.  The first step in showing this is to write $\frac{\dd}{\dd x} \frac { R_{k-1}(x; f, x_0)  } { (x - x_0)^{k} }$ as a weighted integral of $f^{(k)}$.  To do so, we use the following two propositions.

\newcommand{\proprdderiv}{
For any $j$-times differentiable function $f: \reals \to \reals$,
\[
  \frac {\dd}{\dd x} R_j(x; f, x_0) = R_{j-1}(x; f', x_0).
\]
}
\newcommand{\proprdderivproof}{
\begin{align*}
  \frac {\dd}{\dd x} R_j(x; f, x_0)
  & = \frac {\dd}{\dd x} \paren { f(x) - \sum_{i=0}^j \frac {1} {i!} f^{(i)}(x_0) (x - x_0)^i } \\
  & = f'(x) - \sum_{i=1}^j \frac {1} {(i-1)!} f^{(i)}(x_0) (x - x_0)^{i-1} \\
  & = f'(x) - \sum_{i=1}^j \frac {1} {(i-1)!} (f')^{(i-1)}(x_0) (x - x_0)^{i-1} \\
  & = f'(x) - \sum_{h=0}^{j-1} \frac {1} {h!} (f')^{(h)}(x_0) (x - x_0)^{h} \\
  & = R_{j-1}(x; f', x_0).
\end{align*}
}
\begin{proposition} \label{prop:rd_deriv}
\proprdderiv
\end{proposition}

Note that the proposition is valid for $j = 0$, in which case it simply says that $\frac {\dd}{\dd x} R_0(x; f, x_0) = f'(x)$ (recall that we define $R_{-1}(x; f', x_0) \eqdef f'(x)$).  Proposition~\ref{prop:rd_deriv} is straightforward to prove, and a formal proof is given in Appendix A.

We will also use the well-known integral form of the Taylor remainder series \cite{apostol1991calculus}.
\begin{proposition} [Integral form of Taylor remainder series] \label{prop:rd_integral}
For any function $f: \reals \to \reals$, and any integer $j \ge 0$, where $f^{(j)}$ is absolutely continuous over the closed interval between $x_0$ and $x$,
\[
  R_j(x; f, x_0) = \int_{t=x_0}^x f^{(j+1)}(t) \frac {1} {j!} (x -t)^j \dd t.
\]
\end{proposition}

Taking the derivative $\frac{\dd}{\dd x} \frac {R_{k-1}(x; f, x_0)}{(x - x_0)^{k}}$, and applying Propositions~\ref{prop:rd_deriv} and \ref{prop:rd_integral}, we obtain the following lemma, whose proof is given in Appendix A.

\newcommand{\lemratioderivintegral}{
For any function $f: \reals \to \reals$, scalars $x_0, x  \in \reals$, and integer $k \ge 2$, where $f^{(k-1)}$ is absolutely continuous over the closed interval between $x_0$ and $x$,
\[
  \frac{\dd}{\dd x} \frac{R_{k-1}(x; f, x_0)}{(x - x_0)^{k}}
  =  \frac{1}{(x - x_0)^{k+1}} \int_{t=x_0}^x f^{(k)}(t) \weight(t) \dd t
\]
where
\[
  \weight(t) \eqdef (x - x_0) \frac {1} {(k-2)!} (x - t)^{k-2} - \frac{k}{(k-1)!} (x - t)^{k-1}.
\]
}
\newcommand{\lemratioderivintegralproof}{
Using the rule for the derivative of a ratio, then applying Proposition~\ref{prop:rd_deriv}, we have
\begin{align}
  \frac{\dd}{\dd x} \frac {R_{k-1}(x; f, x_0)} {(x - x_0)^{k}}
  & = \frac {\frac{\dd}{\dd x} R_{k-1}(x; f, x_0)} {(x-x_0)^{k}} - k \frac {R_{k-1}(x; f, x_0)} {(x-x_0)^{k+1}} \nonumber \\
  & = \frac {R_{k-2}(x; f', x_0)} {(x-x_0)^{k}} - k \frac {R_{k-1}(x; f, x_0)} {(x-x_0)^{k+1}} \nonumber \\
  & = \frac {1} {(x - x_0)^{k+1}} \paren { (x - x_0) R_{k-2}(x; f', x_0) - k R_k(x; f, x_0) }. \label{eq:ratio_deriv_expr}
\end{align}
Using the integral form of the Taylor remainder series (Proposition~\ref{prop:rd_integral}) to rewrite $R_{k-2}(x; f', x_0)$ and $R_{k-1}(x; f, x_0)$, we have
\begin{align}
  & (x - x_0) R_{k-2}(x; f', x_0) - k R_{k-1}(x; f, x_0) \nonumber \\
  & \quad \quad = (x - x_0) \int_{t=x_0}^x (f')^{(k-1)}(t) \frac {1} {(k-2)!} (x - t)^{k-2} \dd t - \nonumber \\
  & \quad \quad\quad \quad \quad \quad k \int_{t=x_0}^x f^{(k)}(t) \frac {1} {(k-1)!} (x - t)^{k-1} \dd t \nonumber \\
  & \quad \quad = \int_{t=x_0}^x f^{(k)}(t) \weight(t) \dd t. \label{eq:rewrite_as_integral}
\end{align}
Plugging the expression on the right hand side of \eqref{eq:rewrite_as_integral} into \eqref{eq:ratio_deriv_expr} completes the proof.
}
\begin{lemma} \label{lem:ratio_deriv_integral}
\lemratioderivintegral
\end{lemma}

The final and most involved step of the proof is to carefully analyze the sign of the expression given in Lemma~\ref{lem:ratio_deriv_integral} to show that if $f^{(k)}$ is monotonically increasing (decreasing), then $\frac{R_{k-1}(x; f, x_0)}{(x - x_0)^{k}}$ is monotonically increasing (decreasing) as well.  This is proved in Lemma~\ref{lem:monotone_deriv}, from which Theorem~\ref{thm:spoly_monotone} immediately follows.  The proof is given in Appendix A.

\newcommand{\lemmonotonederiv}{
Let $f: \reals \to \reals$ be a function that is $k$ times differentiable over the interval $[a, b]$, for $k \ge 0$.  If $f^{(k)}$ is monotonically increasing (decreasing) over $[a,b]$, then for $x_0 \in (a, b)$, $\frac{R_{k-1}(x; f, x_0)}{(x - x_0)^{k}}$ is monotonically increasing (decreasing) over $[a, b]$.
}
\newcommand{\lemmonotonederivproof}{
For $k = 0$, we have $f^{(k)}(x) = f(x) = \frac{R_{k-1}(x; f, x_0)}{(x - x_0)^{k}}$, and the theorem is trivially true.  For $k = 1$, the proof was given in the main text.
We now consider the general case, $k \ge 2$.

Suppose that $f^{(k)}$ is either monotonically increasing or monotonically decreasing, and define
\begin{equation}
S \eqdef \begin{cases}
1 & \mbox {if } f^{(k)} \mbox{ is monotonically increasing } \\
-1 &\mbox{otherwise.} 
\end{cases}
\end{equation}

To simplify the proof, we will assume that $f^{(k)}$ is differentiable, and that $\frac {\dd} {\dd t} f^{(k)}(t)$ is never exactly 0.  This assumption, together with the assumption that $f^{(k)}$ is either monotonically increasing or decreasing over $[a, b]$, implies that $\sign\paren{\frac{\dd}{\dd t} f^{(k)}(t)}$ is the same for all $t \in [a, b]$, and we have
\begin{equation}
  S = \sign\paren{ \frac {\dd} {\dd t} f^{(k)}(t)} \quad \forall t \in [a, b].
\end{equation}

To prove the lemma, it suffices to show that $\sign\paren{\frac{\dd}{\dd x} \frac{R_{k-1}(x; f, x_0)}{(x - x_0)^{k}}} = S$.  To do so, we first rewrite $\frac{\dd}{\dd x} \frac{R_{k-1}(x; f, x_0)}{(x - x_0)^{k}}$ in terms of $f^{(k)}$.  By Lemma~\ref{lem:ratio_deriv_integral},
\begin{equation} \label{eq:ratio_deriv_expr_from_lemma}
  \frac{\dd}{\dd x} \frac{R_{k-1}(x; f, x_0)}{(x - x_0)^{k}}
  =  \frac{1}{(x - x_0)^{k+1}} \int_{t=x_0}^x f^{(k)}(t) \weight(t) \dd t
\end{equation}
where
\begin{equation} \label{eq:wt_def}
  \weight(t) \eqdef (x - x_0) \frac {1} {(k-2)!} (x - t)^{k-2} - \frac{k}{(k-1)!} (x - t)^{k-1}.
\end{equation}

The bulk of the proof is contained in the following claim.

\noindent\emph{Claim 1.}
\[
  \sign\paren{ \int_{t=x_0}^x f^{(k)}(t) \weight(t) \dd t } = \sign\paren{ (x - x_0)^{k-1} } \cdot S.
\]
\emph{Proof of Claim 1.}
To prove the claim, we first express the sign of $\weight(t)$ in terms of the sign of $t - t^*$, where
\begin{equation}
  t^* \eqdef x - \frac{k-1}{k}(x - x_0).
\end{equation}
Using the definition of $\weight(t)$, and the fact that $\sign(a b) = \sign(a) \sign(b)$ for any $a, b \in \reals$, we have
\begin{align}
  \sign(\weight(t))
  & = \sign((x - t)^{k-2}) \sign \paren { (x - x_0) \frac {1} {(k-2)!} - \frac{k}{(k-1)!} (x - t) } \nonumber \\
  & = \sign((x - t)^{k-2}) \sign \paren { (x - x_0) \frac {k-1} {k} - (x - t) } \nonumber \\
  & = \sign((x - t)^{k-2}) \sign \paren { t - \paren { x - \frac{k-1}{k}(x - x_0) } } \nonumber \\
  & = \sign((x - t)^{k-2}) \sign( t - t^*). \label{eq:sign_wt}
\end{align}
Next, we show that $\int_{t=x_0}^x f^{(k+1)}(t) \weight(t) \dd t$ is equivalent to a different integral whose sign is easier to analyze.  To do so, first note that
\begin{equation}
\int_{t=x_0}^x \weight(t) \dd t
= \paren { (x - x_0) \frac{1}{(k-1)!} (x - t)^{k-1} - \frac{1}{(k-1)!} (x - t)^{k}  } \bigg |_{x_0}^x
= 0.
\end{equation}
Therefore,
\begin{equation} \label{eq:equiv_integral}
  \int_{t=x_0}^x f^{(k)}(t) \weight(t) \dd t = \int_{t=x_0}^x \paren { f^{(k)}(t) - f^{(k)}(t^*) } \weight(t) \dd t.
\end{equation}
We now compute the sign of the integrand on the right hand side of \eqref{eq:equiv_integral}.  To do so, first note that because $f^{(k)}$ is either monotonically increasing or monotonically decreasing,
\begin{equation} \label{eq:sign_deriv_difference}
  \sign \paren { f^{(k)}(t) - f^{(k)}(t^*) } = \sign(t - t^*) \cdot S.
\end{equation}
It follows that, for any $t$ in the closed interval between $x_0$ and $x$,
\begin{align}
  & \sign \paren { (f^{(k)}(t) - f^{(k)}(t^*)) \weight(t) } \nonumber \\
  & \quad \quad = \sign \paren { f^{(k)}(t) - f^{(k)}(t^*) } \sign(\weight(t)) \nonumber \\
  & \quad \quad = \sign(t - t^*) \cdot S \cdot \sign((x - t)^{k-2}) \sign( t - t^*) & \mbox {by \eqref{eq:sign_wt} and \eqref{eq:sign_deriv_difference}} \nonumber \\
  & \quad \quad = \sign((x - t)^{k-2}) \cdot S \nonumber \\
  & \quad \quad = \sign((x - x_0)^{k-2}) \cdot S \label{eq:sign_rewritten_integrand}
\end{align}
where the last step uses the fact that if $t$ is between $x_0$ and $x$, then $x-t$ has the same sign as $x - x_0$.

To complete the proof, note that for any function $h: \reals \to \reals$, if $\sign(h(t)) = C$ for all $t$ in the closed interval between $x_0$ and $x$, then $\sign(\int_{t=x_0}^x h(t) \dd t) = C \cdot \sign(x - x_0)$.
Thus, using \eqref{eq:sign_rewritten_integrand},
\begin{align}
  \sign \paren { \int_{t=x_0}^x (f^{(k)}(t) - f^{(k)}(t^*)) \weight(t) \dd t }
  & = \sign((x - x_0)^{k-2}) \cdot S \cdot \sign(x - x_0) \nonumber \\
  & = \sign((x - x_0)^{k-1}) \cdot S. \label {eq:sign_integral}
\end{align}
Taking the sign of both sides of \eqref{eq:equiv_integral}, and using \eqref{eq:sign_integral} completes the proof of Claim 1.

\noindent$\blacksquare$

To complete the proof, we take the sign of both sides of \eqref{eq:ratio_deriv_expr_from_lemma} and apply Claim 1:
\begin{align}
  & \sign \paren { \frac{\dd}{\dd x} \frac{R_{k-1}(x; f, x_0)}{(x - x_0)^{k+1}}  } \nonumber \\
  & \quad = \sign \paren { \frac{1}{(x - x_0)^{k+1}} \int_{t=x_0}^x f^{(k)}(t) \weight(t) } & \mbox { by \eqref{eq:ratio_deriv_expr_from_lemma}} \nonumber \\
  & \quad = \sign \paren { \frac{1}{(x - x_0)^{k+1}} } \sign \paren { \int_{t=x_0}^x f^{(k)}(t) \weight(t) }  \nonumber \\
  & \quad = \sign \paren { \frac{1}{(x - x_0)^{k+1}} } \sign((x - x_0)^{k-1}) \cdot S & \mbox { by Claim 1 } \nonumber \\
  & \quad = \sign \paren { (x - x_0)^{-2} } \cdot S \nonumber \\
  & \quad = S.
\end{align}
}
\begin{lemma} \label{lem:monotone_deriv}
\lemmonotonederiv
\end{lemma}

\subsubsection{Examples}

We conclude this section by using Theorem~\ref{thm:spoly_monotone} to derive sharp Taylor polynomial enclosures for a handful of elementary functions.

\begin{table*}[h]
        \caption{Sharp degree $k$ Taylor polynomial enclosures for functions with monotone $k$th derivative.}
        \label{tab:spoly_monotone}
  \centering
        \begin{small}
        \begin{sc}
                \begin{tabular}{llll}
    \toprule
          $f$ & Conditions &  $\Istarlep_{k}(f, x_0, [a, b])$ &  $\Istarrep_{k}(f, x_0, [a, b])$ \\
    \midrule

\makecell{$x \mapsto c^x$ \\ ($c > 0$)} & None & $\frac {  c^a - \sum_{i=0}^{k-1} \frac {1} {i!} \ln(c)^i c^{x_0} (a - x_0)^i  } {(a - x_0)^{k}}$ & $\frac {  c^{b} - \sum_{i=0}^{k-1} \frac {1} {i!} \ln(c)^i c^{x_0} (b - x_0)^i  } {(b - x_0)^{k}}$ \\

$\log$ & \makecell[l]{$k=0$,\\$a > 0$} & $\log(a)$ & $\log(b)$ \\

 & \makecell[l]{$k$ odd,\\$a > 0$}& $\frac { \log\paren{\frac{b}{x_0}} - \sum_{i=1}^{k-1} \frac {(-1)^{i+1} (b - x_0)^i } { i  x_0^i }  } {(b - x_0)^{k}}$ & $\frac { \log\paren{\frac{a}{x_0}} - \sum_{i=1}^{k-1} \frac {(-1)^{i+1} (a - x_0)^i} { i  x_0^i }  } {(a - x_0)^{k}}$  \\

 & \makecell[l]{$k$ even,\\$a > 0$}& $\frac { \log\paren{\frac{a}{x_0}} - \sum_{i=1}^{k-1} \frac {(-1)^{i+1} (a - x_0)^i } { i  x_0^i }  } {(a - x_0)^{k}}$ & $\frac { \log\paren{\frac{b}{x_0}} - \sum_{i=1}^{k-1} \frac {(-1)^{i+1} (b - x_0)^i} { i  x_0^i }  } {(b - x_0)^{k}}$  \\

$\mathrm{abs}$ & $k=1$ & $\frac{|a| - |x_0|} {a - x_0}$ & $\frac{|b| - |x_0|} {b - x_0}$ \\

\makecell{$x \mapsto x^c$\\$(c \in \integers)$} & \makecell[l]{$k < c$,\\ $c - k$ odd \\or $a \ge 0$} & $\frac {a^c - \sum_{i=0}^{k-1} {c \choose i} x_0^{c-i} (a-x_0)^i} {(a - x_0)^{k}}$ & $\frac {b^c - \sum_{i=0}^{k-1} {c \choose i} x_0^{c-i} (b-x_0)^i} {(b - x_0)^{k}}$ \\

&\makecell[l]{$k < c$,\\ $c - k$ even,\\ $b \le 0$}  &  $\frac {b^c - \sum_{i=0}^{k-1} {c \choose i} x_0^{c-i} (b-x_0)^i} {(b - x_0)^{k}}$ &  $\frac {a^c - \sum_{i=0}^{k-1} {c \choose i} x_0^{c-i} (a-x_0)^i} {(a - x_0)^{k}}$ \\

\bottomrule
  \end{tabular}
  \end{sc}
  \end{small}
\end{table*}

For $c > 0$, the function $f(x) = c^x$ has $k$th derivative $f^{(k)}(x) = \ln(c)^k c^x$, which is monotonically increasing for all $x$.
Thus, by Theorem~\ref{thm:spoly_monotone},
\begin{align}
  & \Istar_{k}(f, x_0, [a, b]) \nonumber \\
  & \quad = \left [ \frac {  c^a - \sum_{i=0}^{k-1} \frac {1} {i!} \ln(c)^i c^{x_0} (a - x_0)^i  } {(a - x_0)^{k}},  \frac {  c^b - \sum_{i=0}^{k-1} \frac {1} {i!} \ln(c)^i c^{x_0} (b - x_0)^i  } {(b - x_0)^{k}} \right ].
\end{align}

Theorem~\ref{thm:spoly_monotone} can similarly be applied to the $\log$ function, the absolute value function and the function $x \mapsto x^c$, as summarized in Table~\ref{tab:spoly_monotone}.

Of course the theorem can also be applied to more complex functions.  For example, the function $f(x) = \frac 3 2 \exp(3 x) - 25 x^2$ has a monotonically increasing second derivative, and Theorem~\ref{thm:spoly_monotone} can therefore be used to derive the sharp quadratic Taylor polynomial enclosure, illustrated in Figure~\ref{fig:example_enclosure}.  The theorem can also be applied to the $\sin$ and $\cos$ functions so long as the trust region does not contain any integer multiples of $\frac \pi 2$ (which guarantees that the $k$th derivative is monotone over the trust region).

In addition, Theorem~\ref{thm:spoly_monotone} yields as collaries a large number of mathematical inequalities of the type often derived on a one-off basis for use in proofs.
Table~\ref{tab:spoly_monotone_corollaries} gives a handful of examples, each of which appeared on a recently-published cheat sheet of useful mathematical inequalities \cite{kozma2023useful}.

\begin{table*}[h]
        \caption{Some commonly-used mathematical inequalities \cite{kozma2023useful} that can be derived as corollaries of Theorem~\ref{thm:spoly_monotone}.}
        \label{tab:spoly_monotone_corollaries}
  \centering
        \begin{small}
        \begin{sc}
                \begin{tabular}{llll}
    \toprule
          Inequality & Valid for \\
    \midrule

  $1 - x \ln(2) \le 2^{-x} \le 1 - \frac x 2$ & $x \in [0, 1]$ \\

  $e^{-x} \le 1 - \frac x 2$ & $x \in [0, 1.59]$ \\

  $e^x \le 1 + x + x^2$ & $x \le 1.79$ \\

  $x - \frac{x^2}{2} \le \ln(1+x)$ & $x \ge 0$ \\

  $x - x^2 \le \ln(1+x)$ & $x \ge -.68$ \\

  $\frac{x+1}{2} - \frac{(x-1)^2}{2} \le \sqrt{x} \le  \frac{x+1}{2} - \frac{(x-1)^2}{8}$ & $x \in [0, 1]$ \\ 

\bottomrule
  \end{tabular}
  \end{sc}
  \end{small}
\end{table*}

\subsection{Functions with Even-Symmetric Hessian} \label{sec:even_symmetric}

Many commonly-used neural network activation functions, such as $\mathrm{ReLU}$ \cite{nair2010rectified}, $\mathrm{GELU}$ \cite{hendrycks2016gaussian}, SiLU (a.k..a.\ Swish) \cite{elfwing2018sigmoid,ramachandran2017searching} and $\mathrm{Softplus}$, have Hessians that are even-symmetric (meaning $f''(x) = f''(-x)\ \forall x \in \reals$).  In this section we derive Taylor polynomial enclosures for functions with even symmetric Hessians.  Unlike the results in the previous section, here we will only consider quadratic Taylor polynomial enclosures ($k = 2$), which are arguably the most relevant in a neural network optimization context.

Our main result is the following theorem.

\newcommand{\thmevensymmetric}{
Let $f: \reals \to \reals$ be a twice-differentiable function, and suppose for some $\alpha > 0$, $f''$ is even symmetric over $[-\alpha, \alpha]$ (meaning $f''(x) = f''(-x)\ \forall x \in [-\alpha, \alpha]$), and $f''$ is decreasing over $[0, \alpha]$.  Then, the sharp quadratic Taylor polynomial enclosure of $f$ over $[a, b] \subseteq [-\alpha, \alpha]$ at $x_0 \in (a, b)$  is given by:
\[
  \Istar_2(f, x_0, [a, b]) = \left [ \min \set { \frac {  R_1(a; f, x_0)  } {(a - x_0)^2}, \frac {R_1(b; f, x_0)} {(b - x_0)^2} } , \frac { R_1(c; f, x_0) } {(c - x_0)^2} \right ]
\]
where $c \eqdef \min \set { b, \max \set { -x_0, a} }$.
}
\newcommand{\thmevensymmetricproof}{
Consider some fixed $x \in [a, b] \setminus \set{x_0}$, and define
\begin{equation}
  \weight(t) \eqdef 2 t - x - x_0.
\end{equation}
By Lemma~\ref{lem:ratio_deriv_integral},
\begin{equation} \label{eq:ratio_deriv}
\frac{\dd}{\dd x} \frac {R_1(x; f, x_0)} {(x - x_0)^2}
  = \frac {1} {(x - x_0)^3} \int_{t=x_0}^x f''(t) \weight(t) \dd t.
\end{equation}

The bulk of the proof is contained in the following claim.

\noindent\emph{Claim 1.}
\[
  \sign \paren { \int_{t=x_0}^x f''(t) \weight(t) \dd t } = - \sign (x + x_0) \sign (x - x_0).
\]
\noindent\emph{Proof of Claim 1.}
First, observe that for
\begin{equation}
  t^* \eqdef \frac{x + x_0}{2}
\end{equation}
we have
\begin{equation} \label{eq:w_anti_symmetry}
   \weight(t^* + \delta) =  2 \delta = - \weight(t^* - \delta) \quad \forall \delta \in \reals.
\end{equation}
We now use \eqref{eq:w_anti_symmetry} to rewrite the integral on the right hand side of \eqref{eq:ratio_deriv}:
\begin{align}
  & \int_{t=x_0}^x f''(t) \weight(t) \dd t \nonumber \\
  & \quad \quad = \int_{t=x_0}^{t^*} f''(t) \weight(t) \dd t + \int_{t=t^*}^{x} f''(t) \weight(t) \dd t \nonumber \\
  & \quad \quad = \int_{t=x_0}^{t^*} f''(t) \weight(t) \dd t + \int_{\delta=0}^{x - t^*} f''(t^* + \delta) \weight(t^* + \delta) \dd \delta &
(\delta \eqdef t - t^*)
\nonumber \\
  & \quad \quad = \int_{\delta'=0}^{t^* - x_0} f''(t^* - \delta') \weight(t^* - \delta') \dd \delta' + \int_{\delta=0}^{x - t^*} f''(t^* + \delta) \weight(t^* + \delta) \dd \delta &
(\delta' \eqdef t^* - t)
\nonumber \\
  & \quad \quad = \int_{\delta'=0}^{t^* - x_0} - f''(t^* - \delta') \weight(t^* + \delta') \dd \delta' + \int_{\delta=0}^{x - t^*} f''(t^* + \delta) \weight(t^* + \delta) \dd \delta & \mbox { by } \eqref{eq:w_anti_symmetry} \nonumber \\
  & \quad \quad = \int_{\delta=0}^{t^* - x_0} - f''(t^* - \delta) \weight(t^* + \delta) \dd \delta + \int_{\delta=0}^{t^* - x_0} f''(t^* + \delta) \weight(t^* + \delta) \dd \delta \nonumber \\
  & \quad \quad = \int_{\delta=0}^{t^* - x_0} \weight(t^* + \delta) ( f''(t^* + \delta) - f''(t^* - \delta) ) \dd \delta. \label{eq:rewrite_integral}
\end{align}
It is easy to verify that $t^* + \delta \in [a, b]$ and $t^* - \delta \in [a, b]$.  Therefore, by Lemma~\ref{lem:even_symmetric_sign},
\begin{equation} \label{eq:sign_difference}
  \sign(f''(t^* + \delta) - f''(t^* - \delta)) = - \sign(t^*) \sign(\delta).
\end{equation}

We are now ready to analyze the sign of the integral on the right hand side of \eqref{eq:rewrite_integral}.  First, we show that the integrand has a constant sign (independent of $\delta$):
\begin{align}
  & \sign \paren { \weight(t^* + \delta) ( f''(t^* + \delta) - f''(t^* - \delta) ) } \nonumber \\
  & \quad \quad = \sign\paren{\weight(t^* + \delta)} \sign\paren{f''(t^* + \delta) - f''(t^* - \delta)} \nonumber \\
  & \quad \quad = \sign(2 \delta) \sign\paren{f''(t^* + \delta) - f''(t^* - \delta)} \nonumber \\
  & \quad \quad = - \sign(2 \delta) \sign(t^*) \sign(\delta) &\mbox { by } \eqref{eq:sign_difference} \nonumber \\
  & \quad \quad = - \sign(t^*) \nonumber \\
  & \quad \quad = -\sign(x + x_0). \label{eq:sign_integrand}
\end{align}

To complete the proof, note that for any function $h: \reals \to \reals$, if $\sign(h(t)) = C$ for all $t$ in the closed interval between $0$ and $t^* - x_0$, then $\sign(\int_{t=x_0}^x h(t) \dd t) = C \cdot \sign(t^* - x_0)$.  Thus, using \eqref{eq:sign_integrand}, we have
\begin{align}
  \sign \paren {   \int_{\delta=0}^{t^* - x_0} \weight(t^* + \delta) ( f''(t^* + \delta) - f''(t^* - \delta) ) \dd \delta }
  & = - \sign(x + x_0) \sign(t^* - x_0) \nonumber \\
  & = - \sign(x + x_0) \sign (x - x_0). \label{eq:sign_rewritten_integral}
\end{align}
Combining \eqref{eq:rewrite_integral} and \eqref{eq:sign_rewritten_integral} completes the proof of Claim 1.

\noindent$\blacksquare$

To complete the proof, we use \eqref{eq:ratio_deriv} and Claim 1 to show that for any $x \in [a, b] \setminus \set { x_0 }$:
\begin{align} 
\sign\paren{\frac{\dd}{\dd x} \frac {R_1(x; f, x_0)} {(x - x_0)^2}}
& = \sign \paren { \frac {1} {(x - x_0)^3} \int_{t=x_0}^x f''(t) \weight(t) \dd t } & \mbox { by } \eqref{eq:ratio_deriv} \nonumber \\
& = \sign \paren { \frac {1} {(x - x_0)^3 } } \sign \paren { \int_{t=x_0}^x f''(t) \weight(t) \dd t } \nonumber \\
& = - \sign \paren { \frac {1} {(x - x_0)^3 } } \sign(x + x_0) \sign (x - x_0) & \mbox { by Claim 1} \nonumber \\
& = -\sign(x + x_0) \label{eq:sign_ratio_deriv}.
\end{align}
It follows from \eqref{eq:sign_ratio_deriv} that the function $r(x) \eqdef \frac {R_1(x; f, x_0)} {(x - x_0)^2}$ is increasing for $x < -x_0$, decreasing for $x > -x_0$, and locally maximized at $x = x_0$.  Thus, over the interval $[a, b]$, $r(x)$ is maximized at the point $x = \min \set{b, \max \set {-x_0, a }}$, and minimized at either $x = a$ or $x = b$.  Using the definition of $\Istar_2(f, x_0, [a, b])$ then proves the lemma.
}
\begin{theorem} \label{thm:even_symmetric}
\thmevensymmetric
\end{theorem}

To prove Theorem~\ref{thm:even_symmetric}, we analyze the minimum and maximum possible values of the ratio $r(x) \eqdef \frac {R_1(x; f, x_0)} {(x - x_0)^2}$ for $x \in [a, b]$, which by definition are the end points of the interval $\Istar_2(f, x_0, [a, b])$.  To do so, we first use Lemma~\ref{lem:ratio_deriv_integral} to write $r'(x)$ in terms of a weighted integral of $f''(t)$, and then carefully analyze the sign of the resulting expression in order to show that $r(x)$ is maximized at $x = -x_0$, decreasing for $x > -x_0$, and increasing for $x < -x_0$.  This implies that the minimum value of $r(x)$ for $x \in [a, b]$ is either $r(a)$ or $r(b)$ (whichever is smaller), while the maximum value is $r(c)$, where $c$ is the point in $[a, b]$ that is closest to $-x_0$, yielding the expression given in Theorem~\ref{thm:even_symmetric}.  The full proof is given in Appendix A.


Figure~\ref{fig:activation_functions} plots a number of commonly-used activation functions whose Hessian is even symmetric.

For the $\mathrm{Hard\ SiLU}$, $\mathrm{Leaky ReLU}$, $\mathrm{ReLU}$, and $\mathrm{Softplus}$ functions, the Hessian is monotonically decreasing for $x \ge 0$, and Theorem~\ref{thm:even_symmetric} can be used to derive a sharp quadratic Taylor polynomial enclosure for any trust region $[a, b]$.  For $\mathrm{GELU}$ and $\mathrm{SiLU}$ (a.k.a.\ $\mathrm{Swish}$), Theorem~\ref{thm:even_symmetric} can be used so long as the trust region $[a, b]$ does not include the regions far from the origin in which the Hessian is non-monotonic (for trust regions not satisfying this condition, bounds on the Hessian can be used to derive a non-sharp quadratic Taylor polynomial enclosure, as discussed in \S\ref{sec:spoly_baseline}).

\begin{figure}[h]
\begin{center}
\includegraphics[width=\sidebysidefigwidth]{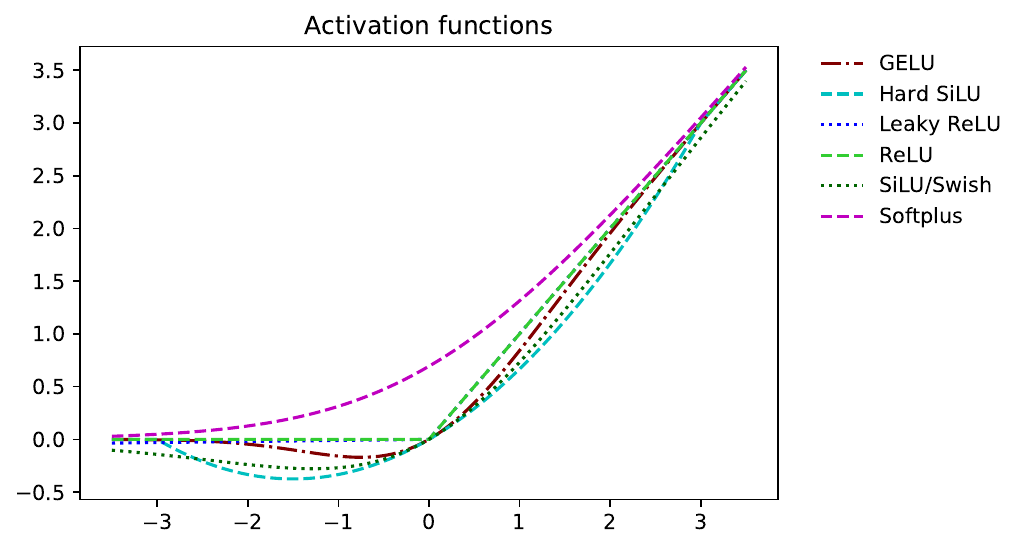}
\includegraphics[width=\sidebysidefigwidth]{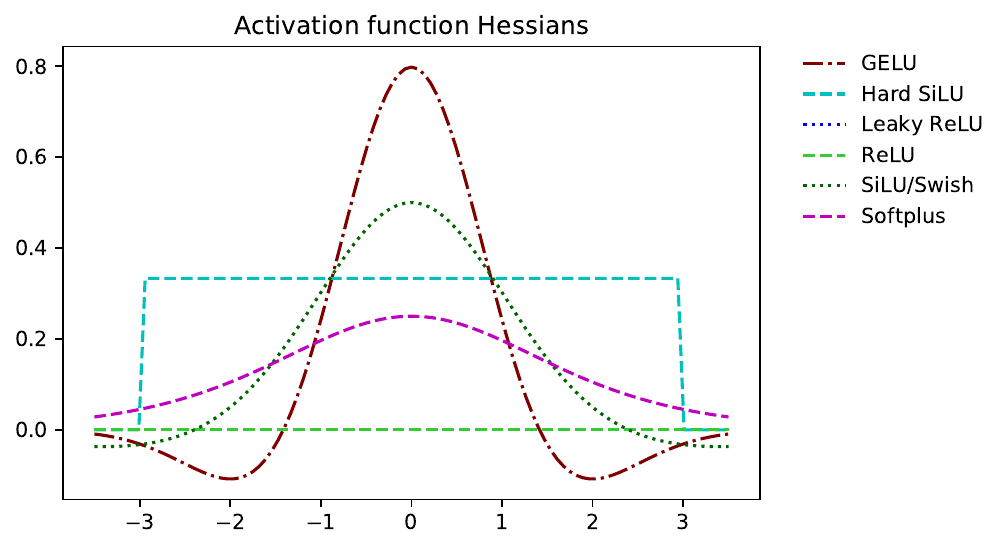}
\caption{Some commonly-used neural network activation functions (left) and their Hessians (right).}
\label{fig:activation_functions}
\end{center}
\end{figure}

\section{Improvement over Classical Baseline} \label{sec:comparison}

How much better are sharp Taylor polynomial enclosures than the non-sharp Taylor polynomial enclosures discussed in \S\ref{sec:spoly_baseline}, obtained via bounds on the range of the $k$th derivative?

In general, the answer to this question depends on the function being bounded.  However, as the size of the trust region approaches zero, it turns out that the interval defining the sharp Taylor polynomial enclosure is always at least $k+1$ times narrower than an interval based on the range of the $k$th derivative, except in the trivial case where the function being bounded is a polynomial of degree $k$ or smaller (in which case both intervals have zero width).  This is established in the following theorem, whose proof is given in Appendix~\ref{sec:proofs}.

\newcommand{\thmsharpadvantage}{
For any integer $k > 0$, any analytic function $f: \reals \to \reals$ that is not a polynomial of degree $k$ or smaller, and any $x_0 \in \reals$,
\[
  \lim_{\epsilon \rightarrow 0} \frac { \wid(\Idb_{k}(f, [x_0, x_0 + \epsilon]) ) } { \wid(\Istar_{k}(f, x_0, [x_0, x_0 + \epsilon] ) )} = {k + \ell \choose \ell}
\]
where $\Istar_k$ is the interval defining the sharp Taylor polynomial enclosure (see Definition~\ref{def:spoly}),
\[
  \Idb_{k}(f, [a, b]) \eqdef \frac {1} {k!} \left [\inf_{y \in [a, b]} \set { f^{(k)}(y) }, \sup_{y \in [a, b]} \set { f^{(k)}(y) } \right ],
\]
$\ell \ge 1$ is the smallest positive integer such that $f^{(k+\ell)}(x_0) \neq 0$,
and for any interval $\gi = [\gilep, \girep]$, we define $\wid(\gi) \eqdef \girep - \gilep$.

Furthermore, both $\Idb_{k}(f, [x_0, x_0 + \epsilon])$ and $\Istar_{k}(f, x_0, [x_0, x_0 + \epsilon])$ have width $O(\epsilon^\ell)$ (as $\epsilon \to 0$).
}
\newcommand{\thmsharpadvantageproof}{
Because $f$ is analytic,
\begin{equation}
  f(x) = \sum_{i=0}^\infty \frac {1} {i!} f^{(i)}(x_0) (x - x_0)^i.
\end{equation}
Therefore,
\begin{equation} \label{eq:f_k}
  f^{(k)}(x) = \sum_{i=k}^\infty \frac {1} {(i-k)!} f^{(i)}(x_0) (x - x_0)^{i-k}.
\end{equation}
For sufficiently small $\epsilon$, $f^{(k)}$ must be either monotonically increasing or monotonically decreasing over the interval $[x_0, x_0 + \epsilon]$.  Assume without loss of generality that $f^{(k)}$ is monotonically increasing over $[x_0, x_0 + \epsilon]$.  Under this assumption,
\begin{equation}
  \inf_{y \in [a, b]} \set { f^{(k)}(y) } = f^{(k)}(x_0)
\end{equation}
and
\begin{align}
  \sup_{y \in [a, b]} \set { f^{(k)}(y) }
  & = f^{(k)}(x_0 + \epsilon) \nonumber \\
  & =  f^{(k)}(x_0) + \sum_{i=k+1}^\infty \frac {1} {(i-k)!} f^{(i)}(x_0) \epsilon^{i-k} & \mbox{by } \eqref{eq:f_k} \nonumber \\
  & =  f^{(k)}(x_0) + \sum_{i=k+\ell}^\infty \frac {1} {(i-k)!} f^{(i)}(x_0) \epsilon^{i-k} & \mbox {by definition of $\ell$}.
\end{align}
Therefore,
\begin{equation} \label{eq:Idb_width}
   \wid( \Idb_{k}(f, [x_0, x_0 + \epsilon]) ) = \frac {1} {k!} \sum_{i=k+\ell}^\infty \frac {1} {(i-k)!} f^{(i)}(x_0) \epsilon^{i-k}.
\end{equation}
Dividing both sides of \eqref{eq:Idb_width} by $\epsilon^\ell$, and taking the limit as $\epsilon \to 0$, we have
\begin{equation} \label{eq:derivative_bound_ratio}
  \lim_{\epsilon \rightarrow 0} \frac { \wid( \Idb_{k}(f, [x_0, x_0 + \epsilon]) ) } {\epsilon^\ell} = \frac {1} {k! \ell!} f^{(k+\ell)}(x_0).
\end{equation}
At the same time, by Corollary~\ref{cor:spoly_monotone},
\begin{equation}
  \Istar_{k}(f, x_0, [x_0, x_0 + \epsilon]) = \left [ \frac {1} {k!} f^{(k)}(x_0), \frac { R_{k-1}(x_0 + \epsilon; f, x_0) } { \epsilon^{k} }  \right ].
\end{equation}
Therefore,
\begin{align} \label{eq:Istar_width}
  \wid( \Istar_{k}(f, x_0, [x_0, x_0 + \epsilon]) )
  & = \frac { R_{k-1}(x_0 + \epsilon; f, x_0) } { \epsilon^{k} } - \frac {1} {k!} f^{(k)}(x_0) \nonumber \\
  & = \paren { \sum_{i=k}^\infty \frac {1} {i!} f^{(i)}(x_0) \epsilon^{i-k} } - \frac {1} {k!} f^{(k)}(x_0) \nonumber \\
  & = \sum_{i=k+1}^\infty \frac {1} {i!} f^{(i)}(x_0) \epsilon^{i-k} \nonumber \\
  & = \sum_{i=k+\ell}^\infty \frac {1} {i!} f^{(i)}(x_0) \epsilon^{i-k}. \\
\end{align}
Dividing both sides of \eqref{eq:Istar_width} by $\epsilon^\ell$, and taking the limit as $\epsilon \to 0$, we have
\begin{equation} \label{eq:sharp_ratio}
   \lim_{\epsilon \rightarrow 0} \frac { \wid( \Istar_{k}(f, x_0, [x_0, x_0 + \epsilon]) ) } {\epsilon^\ell} = \frac {1} {(k+\ell)!} f^{(k+\ell)}(x_0).
\end{equation}
By definition of $\ell$, the right hand side of \eqref{eq:sharp_ratio} is nonzero.  Dividing \eqref{eq:derivative_bound_ratio} by \eqref{eq:sharp_ratio} then proves the first part of the theorem.

The second part of the theorem (that both intervals have width $O(\epsilon^\ell)$ as $\epsilon \to 0$) follows immediately from \eqref{eq:derivative_bound_ratio} and \eqref{eq:sharp_ratio}.
}
\begin{theorem} \label{thm:sharp_advantage}
\thmsharpadvantage
\end{theorem}

\section{Taylor Polynomial Enclosures for Arbitrary Functions} \label {sec:arbitrary}

We now discuss three ways to compute Taylor polynomial enclosures for arbitrary functions.  The first two methods apply only to one-dimensional functions.  Both of these methods are never worse than the classical baseline discussed in \S\ref{sec:spoly_baseline}, and return the sharp Taylor polynomial enclosure once the trust region becomes sufficiently small.  The final method is presented in concurrent work \cite{streeter2023automatically}, and applies to arbitrary functions.

\subsection{Using Knowledge of Local Extrema} \label{sec:local_extrema}

For certain functions, such as $\sin$ and $\cos$, it is straightforward to enumerate all local extrema within an arbitrary trust region $[a, b]$.  Given a function $f: \reals \to \reals$, and an oracle that returns all local extrema of $f^{(k)}$ within a given trust region, we can compute a degree $k$ Taylor polynomial enclosure as follows:
\begin{itemize}
  \item If $f^{(k)}$ has no local extrema within the trust region, then $f^{(k)}$ is either monotonically increasing or monotonically decreasing, and we can invoke Theorem~\ref{thm:spoly_monotone} to obtain the sharp enclosure.
  \item If the trust region contains one or more local extrema, we can use knowledge of the local extrema to bound the range of the $k$th derivative, then plug this into the classical bound discussed in \S\ref{sec:spoly_baseline}.
\end{itemize}

This procedure always yields a valid Taylor polynomial enclosure, and yields the sharp Taylor polynomial enclosure for sufficiently small trust regions, as shown in the following corollary of Theorem~\ref{thm:spoly_monotone}.

\newcommand{\Ile}{I^{\mathrm{LocalExtrema}}}
\begin{corollary}
Let $f: \reals \to \reals$ be an analytic function, and denote the set of local extrema of $f^{(k)}$ within the interval $[a, b]$ by $X_k([a, b]) \eqdef \set{x \in [a, b]: f^{(k+1)}(x) = 0}$.  Define the interval
\begin{equation}
  \tilde \gi_k(f, x_0, [a, b]) \eqdef \begin{cases}
    \left [ \frac {R_{k-1}(a; f, x_0)} {(a - x_0)^{k}},  \frac {R_{k-1}(b; f, x_0)} {(b - x_0)^{k}} \right ] & \mbox{if } X(f^{(k)}, [a, b]) = \emptyset \wedge f^{(k)}(b) \ge f^{(k)}(a) \\
    \left [ \frac {R_{k-1}(b; f, x_0)} {(b - x_0)^{k}},  \frac {R_{k-1}(a; f, x_0)} {(a - x_0)^{k}} \right ] & \mbox{if } X(f^{(k)}, [a, b]) = \emptyset \wedge f^{(k)}(b) < f^{(k)}(a) \\
    \frac{1}{k!} \left [ l_k([a, b]), u_k([a,b]) \right ] & \mbox{otherwise}
\end{cases}
\end{equation}
where $l_k([a, b]) = \min_{x \in X_k([a, b]) \cup \set{a, b}} \set { f^{(k)}(x) }$ is the minimum value of $f^{(k)}(x)$ for $x \in [a, b]$, and $u_k([a, b]) = \max_{x \in X_k([a, b]) \cup \set{a, b}} \set { f^{(k)}(x) }$ is the maximum value.

Then, for any $[a, b] \in \intervals$ and any $x_0 \in [a, b]$ the interval $\tilde \gi_k(f, x_0, [a, b])$ defines a Taylor polynomial enclosure of $f$ at $x_0$ over $[a, b]$.
Furthermore, for any $x_0 \in \reals$, there exists a $\delta > 0$ such that
\[
  \tilde \gi_k(f, x_0, [x_0, x_0 + \epsilon]) = \Istar_k(f, x_0, [x_0, x_0 + \epsilon]) \quad \forall \epsilon \in [0, \delta]
\]
where $\Istar_k$ is the interval that defines the sharp degree $k$ Taylor polynomial enclosure (see Definition~\ref{def:spoly}).
\end{corollary}

\subsection{Using Interval Evaluation of the $(k+1)$st Derivative}

In cases where it is not practical to enumerate all local extrema of $f^{(k)}$, one can instead first obtain a formula $\mathsf{f}^{(k+1)}$ for $f^{(k+1)}$ (e.g., using automatic differentiation),
and then evaluate $\mathsf{f}^{(k+1)}([a, b])$ using interval arithmetic.\footnote{For a function $g: \reals \to \reals$ whose value is given by a formula $\mathsf{g}$, and an interval $\gi \in \intervals$, $\mathsf{g}(\gi)$ denotes the interval obtained by plugging $\gi$ into the formula $\mathsf{g}$ and evaluating the resulting expression using the rules of interval arithmetic.}  If the resulting interval $\gi$ contains only non-negative values, then $f^{(k)}$ is monotonically increasing, and we can obtain the sharp Taylor polynomial enclosure via Theorem~\ref{thm:spoly_monotone}.  We can similarly return the sharp Taylor polynomial enclosure if $\gi$ contains only non-positive values.  Finally, if $0 \in \gi$, we can fall back to the classical baseline method discussed in \S\ref{sec:spoly_baseline}.

The following corollary of Theorem~\ref{thm:spoly_monotone} shows that, like the method described in \S\ref{sec:local_extrema}, this procedure always returns a valid Taylor polynomial enclosure, and returns the sharp Taylor polynomial enclosure when the trust region is sufficiently small.

\newcommand{\corasymptoticallyoptimal}{
Let $f: \reals \to \reals$ be an analytic function given by a formula $\mathsf{f}$, and define
\begin{equation}
  \Ias_k(f, x_0, [a, b]) \eqdef \begin{cases}
  \frac {1} {k!} f^{(k)}([a, b]) & \mbox { if } 0 \in \mathsf{f}^{(k+1)}([a, b]) \\
  \left [ \frac {R_{k-1}(a; f, x_0)} {(a - x_0)^{k}},  \frac {R_{k-1}(b; f, x_0)} {(b - x_0)^{k}} \right ] & \mbox { if } \lep{ \mathsf{f}^{(k+1)}([a, b]) } > 0 \\
  \left [ \frac {R_{k-1}(b; f, x_0)} {(b - x_0)^{k}},  \frac {R_{k-1}(a; f, x_0)} {(a - x_0)^{k}} \right ] & \mbox { otherwise} \\
\end{cases}
\end{equation}
where we use $\lep{ \mathsf{f}^{(k+1)}([a, b]) }$ to denote the left endpoint of $\mathsf{f}^{(k+1)}([a, b])$.

Then, for any integer $k > 0$, trust region $[a, b] \in \intervals$, and $x_0 \in [a, b]$, the interval $\Ias_k(f, x_0, [a, b])$ defines a degree $k$ Taylor polynomial enclosure of $f$ at $x_0$ over $[a, b]$.  This enclosure is always at least as tight as the one given by the baseline method described in \S\ref{sec:spoly_baseline}:
\[
  \Ias_k(f, x_0, [a, b]) \subseteq \frac {1} {k!} f^{(k)}([a, b]).
\]
Furthermore, if $f^{(k+1)}(x_0) \neq 0$, then there exists a $\delta > 0$ such that
\[
  \Ias_k(f, x_0, [x_0, x_0 + \epsilon]) = \Istar_k(f, x_0, [x_0, x_0 + \epsilon]) \quad \forall \epsilon \in [0, \delta]
\]
where $\Istar_k$ is the interval that defines the sharp degree $k$ Taylor polynomial enclosure (see Definition~\ref{def:spoly}).
}
\newcommand{\corasymptoticallyoptimalproof}{
We first prove that $I \eqdef \Ias_k(f, x_0, [a, b])$ defines a degree $k$ Taylor polynomial enclosure of $f$ at $x_0$ over $[a, b]$.  Let $J \eqdef \frac {1} {k!} f^{(k+1)}([a, b])$.  As mentioned in \S\ref{sec:spoly_baseline}, the Lagrange form of the Taylor remainder series can be used to show that $I$ defines a Taylor polynomial enclosure of $f$ at $x_0$ over $[a, b]$, or equivalently
\begin{equation} \label{eq:classical}
  I^*_k(f, x_0, [a, b]) \subseteq J.
\end{equation}
If $0 \in J$, then $I = J$, and \eqref{eq:classical} implies that $I$ defines a Taylor polynomial enclosure of $f$ at $x_0$ over $[a, b]$.  Otherwise, if $I \neq J$, then Theorem~\ref{thm:spoly_monotone} implies $I = I^*_k(f, x_0, [a, b])$.  This shows that $I$ always defines a Taylor polynomial enclosure of $f$ at $x_0$ over $[a, b]$, and furthermore that $I \subseteq J$.

To prove the last part of the theorem, it suffices to show that if $f^{(k+1)}(x_0) \neq 0$, then for some sufficiently small $\epsilon > 0$, we have $0 \notin f^{(k+1)}([x_0, x_0 + \epsilon])$.  This follows from the linear convergence of interval arithmetic.
}
\begin{corollary} \label{cor:asymptotically_optimal}
\corasymptoticallyoptimal
\end{corollary}

\subsection{Composing Taylor Polynomial Enclosures}

In concurrent work \cite{streeter2023automatically}, we show that it is possible to compute Taylor polynomial enclosures recursively, combining Taylor polynomial enclosures for simple functions in order to obtain enclosures for more complex functions.  The resulting algorithm, called AutoBound, uses the sharp Taylor polynomial enclosures derived in this work wherever possible.  AutoBound handles arbitrary multivariate functions composed of known univariate functions (such as $\exp$ and $\log$) as well as known bilinear operations (such as matrix multiplications and convolutions).  Among other applications, it can be used to derive majorization-minimization optimizers for the losses used to train deep neural networks.

\section{Conclusions and Future Work} \label{sec:spoly_future}

In this work we derived sharp Taylor polynomial enclosures for functions whose $k$th derivative is monotonically increasing or monotonically decreasing over the trust region, and sharp \emph{quadratic} Taylor polynomial enclosures for functions whose Hessian is even-symmetric.  Sharp enclosures improve on a classical baseline that uses bounds on the range of the $k$th derivative to obtain degree $k$ bounds, by a factor of at least $k+1$ asymptotically (as the trust region width approaches zero).  We also presented two methods for computing Taylor polynomial enclosures for arbitrary functions, both of which produce sharp results asymptotically.

Our results suggest several possibilities for future work.  Perhaps the most obvious one is to develop theory that applies to additional one-dimensional functions.  For example, we do not currently have a way to compute sharp Taylor polynomial enclosures for periodic functions, such as sine or cosine (except in special cases where the trust region is sufficiently small).  It would also be nice to generalize Theorem~\ref{thm:even_symmetric} to Taylor polynomial enclosures of degree $> 2$.  Finally, it would be very interesting to derive sharp Taylor polynomial enclosures for multivariate functions, perhaps building on the results of Browne and McNicholas \cite{browne2015multivariate}.

\ifarxiv
\bibliographystyle{plainnat}
\else
\bibliographystyle{plain}
\fi

\bibliography{sharp_taylor_1d}

\begin{appendices}

\section{Proofs} \label{sec:proofs}

We first prove Proposition~\ref{prop:unique_sharp}.  To do so, we will make use of the fact the product of a scalar and an interval satisfies an associative property.  Recall that for an interval $\gi \in \intervals$ and scalar $\alpha \in \reals$, we defined $\gi \alpha \eqdef  \set{\alpha z: z \in \gi}$.  It follows that for any interval $\gi \in \intervals$ and scalars $\alpha, \beta \in \reals$,
\begin{equation} \label{eq:scalar_interval_product_associative}
  (\gi \alpha) \beta = \gi (\alpha \beta).
\end{equation}

\begin{customproposition}{1}
\propuniquesharp
\end{customproposition}
\begin{proof}
\propuniquesharpproof
\end{proof}

\subsection{Proof of Theorem~\ref{thm:spoly_monotone}}

The idea of the proof will be to identify conditions under which the ratio $\frac{ R_{k-1}(x; f, x_0) }{ (x - x_0)^k }$ is increasing or decreasing.  To do so, we first derive an expression for the derivative of this ratio with respect to $x$.  In doing so we will make use of the following proposition, which provides an expression for the derivative of the numerator of the ratio.

\begin{customproposition}{2}
\proprdderiv
\end{customproposition}
\begin{proof}
\proprdderivproof
\end{proof}

With this proposition in hand, we can now derive an expression for $\frac{\dd}{\dd x} \frac {R_{k-1}(x; f, x_0)} {(x - x_0)^k}$.

\begin{customlemma}{1}
\lemratioderivintegral
\end{customlemma}
\begin{proof}
\lemratioderivintegralproof
\end{proof}

Lastly, we prove Lemma~\ref{lem:monotone_deriv}, from which Theorem~\ref{thm:spoly_monotone} immediately follows.

\begin{customlemma}{2}
\lemmonotonederiv
\end{customlemma}
\begin{proof}
\lemmonotonederivproof
\end{proof}

\subsection{Proof of Theorem~\ref{thm:even_symmetric}}

We now prove Theorem~\ref{thm:even_symmetric}.  We will need the following technical lemma.

\begin{lemma} \label{lem:even_symmetric_sign}
Let $\alpha > 0$ be a positive real number, and let $\gi = [-\alpha, \alpha]$ be an interval.  Let $h: \reals \to \reals$ be a function that is even-symmetric over $\gi$ (meaning $h(x) = h(-x)\ \forall x \in \gi$), and decreasing over $[0, \alpha]$.  Then, for any $t, \delta \in \reals$ such that $t+\delta \in \gi$ and $t - \delta \in \gi$,
\[
  \sign(h(t+\delta)) - \sign(h(t-\delta)) = -\sign(t) \sign(\delta).
\]
\end{lemma}
\begin{proof}
Because $h$ is even symmetric over $\gi$,
\begin{equation} \label{eq:h_abs_y}
  h(y) = h(|y|) \quad \forall y \in \gi.
\end{equation}
Because $h$ is decreasing over $[0, \alpha]$,
\begin{equation} \label{eq:h_sign}
  \sign(h(|y|) - h(|x|)) = - \sign(|y| - |x|) \quad \forall x, y \in \gi.
\end{equation}
Thus,
\begin{align}
\sign(h(t+\delta)) - \sign(h(t-\delta))
& = \sign(h(|t+\delta|)) - \sign(h(|t-\delta|)) & \mbox{by \eqref{eq:h_abs_y}} \nonumber \\
& = - \sign\paren{ |t+\delta| - |t-\delta| }. & \mbox{by \eqref{eq:h_sign}} \label{eq:sign_h_diff}
\end{align}
To complete the proof, it suffices to show
\begin{equation} \label{eq:sign_product}
  \sign(|t + \delta| - |t - \delta|) = \sign(t) \sign(\delta).
\end{equation}
We prove \eqref{eq:sign_product} by considering three cases.
If either $t$ or $\delta$ is 0, then both sides of \eqref{eq:sign_product} are 0.  Otherwise, if $t$ and $\delta$ have the same sign, then $|t + \delta| > |t - \delta|$, and both sides of \eqref{eq:sign_product} are 1.  Finally, if $t$ and $\delta$ have opposite sign, then $|t + \delta| < |t - \delta|$, and both sides of \eqref{eq:sign_product} are -1.
\end{proof}

\ifarxiv
\begin{customthm}{2}
\else
\begin{customthm}{5.2}
\fi
\thmevensymmetric
\end{customthm}
\begin{proof}
\thmevensymmetricproof
\end{proof}


\subsection{Proof of Theorem~\ref{thm:sharp_advantage}}

To prove Theorem~\ref{thm:sharp_advantage}, we will need the following corollary of Theorem~\ref{thm:spoly_monotone}.
\begin{corollary} \label{cor:spoly_monotone}
Let $f: \reals \to \reals$ be an analytic function.
If $f^{(k)}$ is monotonically increasing over an interval $[a, b]$ with $b > a$, then the degree $k$ sharp Taylor polynomial enclosure of $f$ at $a$ over $[a, b]$ is given by $\Istar_{k}(f, a, [a, b]) = \left [ \frac {1} {k!} f^{(k)}(a) ,  \frac {R_{k-1}(b; f, a)} {(b - a)^{k}} \right ]$.
\end{corollary}
\begin{proof}
By Theorem~\ref{thm:spoly_monotone}, for any $\epsilon \in (0, b-a)$ we have
\begin{equation}
  \Istarlep_k(f, a + \epsilon, [a, b]) = \frac {R_{k-1}(a; f, a + \epsilon)} {(-\epsilon)^{k}}.
\end{equation}
Because $f$ is analytic, $\Istarlep_k(f, x_0, [a, b])$ is continuous as a function of $x_0$.  Therefore,
\begin{align*}
  \Istarlep_k(f, a, [a, b])
  & = \lim_{\epsilon \to 0} \Istarlep_k(f, a + \epsilon, [a, b]) \nonumber \\
  & = \lim_{\epsilon \to 0} \frac {R_{k-1}(a; f, a + \epsilon)} {(-\epsilon)^{k}} \nonumber \\
  & = \lim_{\epsilon \to 0} \frac {\sum_{i=k}^\infty \frac {1} {i!} f^{(i)}(x_0) (-\epsilon)^i } {(-\epsilon)^{k}} \nonumber \\
  & = \frac {1} {k!} f^{(k)}(x_0).
\end{align*}
A similar argument shows $\Istarrep_k(f, a, [a, b]) = \frac {R_{k-1}(b; f, a)} {(b - a)^{k}}$, completing the proof.
\end{proof}

\ifarxiv
\begin{customthm}{3}
\else
\begin{customthm}{6.1}
\fi
\thmsharpadvantage
\end{customthm}
\begin{proof}
\thmsharpadvantageproof
\end{proof}

\end{appendices}

\end{document}